\documentclass{amsart}
\usepackage{amsmath, amsthm, amscd, amssymb, amsfonts, latexsym}
\usepackage{fullpage}
\usepackage{bm}
\usepackage{graphicx}
\usepackage[all]{xy}

\newcommand{\F}{\mathbb F}
\newcommand{\thet}{\theta_j(f, \psi)}
\newcommand{\alp}{\alpha_j(f, \psi)}

\DeclareMathOperator{\tr}{tr}

\newtheorem{lem}{Lemma}[section]
\newtheorem{prop}[lem]{Proposition}
\newtheorem{thm}[lem]{Theorem}

\newtheorem{cor}[lem]{Corollary}

\theoremstyle{definition}

\newtheorem{rem}[lem]{Remark}

\date\today
\title{Distribution of zeta zeroes of Artin--Schreier covers}
\author{Alina Bucur, Chantal David, Brooke Feigon, Matilde Lal\'{i}n, Kaneenika Sinha}
\address{Alina Bucur: Department of Mathematics, University of California at San Diego, 
9500 Gilman Drive $\#$0112, La Jolla, CA 92093, USA} \email{alina@math.ucsd.edu}
\address{Chantal David: Department of Mathematics and Statistics,                                                                                 
Concordia University,                                                                                                                       
1455 de Maisonneuve West,                                                                                                                   
Montreal, QC H3G 1M8, Canada} \email{cdavid@mathstat.concordia.ca}
 \address{Brooke Feigon: Department of Mathematics,
The City College of New York, 
CUNY, 
NAC 8/133,
New York, NY 10031, USA} \email{bfeigon@ccny.cuny.edu }
 \address{Matilde Lal\'in: D\'epartement de math\'ematiques et de statistique,
                                    Universit\'e de Montr\'eal.
                                    CP 6128, succ. Centre-ville.
                                     Montreal, QC H3C 3J7, Canada} \email{mlalin@dms.umontreal.ca}
                                     
\address{Kaneenika Sinha:
Department of Mathematics, Indian Institute of Science Education and Research (IISER), first floor, Central Tower, Sai Trinity Building, Garware Circle, Sutarwadi, Pashan, Pune, Maharashtra 411021, India}\email{kaneenika@gmail.com}

\begin{document}

\begin{abstract}
We study the distribution of the zeroes of the zeta functions of the family of Artin-Schreier covers of the projective line over $\mathbb{F}_q$ when $q$ is fixed and the genus goes to infinity. We consider both the global and the mesoscopic regimes, proving that when the genus goes to infinity, the number of zeroes with angles in a prescribed non-trivial subinterval of $[-\pi,\pi)$ has a standard Gaussian distribution (when properly normalized).
\end{abstract}



\maketitle

\section{Introduction}
Recently there has been a great deal of interest in statistics for numbers of rational points on 
curves over finite fields, where the curve varies in a certain family but is always defined over 
a fixed finite field. This is in contrast to situations studied using Deligne's equidistribution theorem \cite{deligne, deligne2}, which requires the size of the finite field to go to infinity, and which tends to produce
statistics related to random matrices in certain monodromy groups. When one fixes the base field, one instead tends to encounter discrete probabilities, typically sums of 
independent identically distributed random variables. The first result in this direction
is the work of Kurlberg and Rudnick for hyperelliptic curves \cite{kr};
other cases considered include cyclic $p$-fold covers of the projective line \cite{bdfl1, bdfl2} (for a slightly different approach see \cite{x2}),
plane curves \cite{bdfl3}, complete intersections in projective spaces \cite{bk},
and general trigonal curves \cite{wood}.

The number of rational points on a curve over a finite field is determined by the zeta function,
and statistical properties of the number of points may be interpreted as properties of the
coefficients of the zeta function. 
A related but somewhat deeper question is to consider statistical properties of \textit{zeroes}
of the zeta function. In the case of hyperelliptic curves, these properties were studied by
Faifman and Rudnick \cite{fr}. A related family was studied in \cite{x1}. 

In this paper, we make similar considerations for
the family of \emph{Artin-Schreier covers of $\mathbb P^1$}; this family is interesting because the
characteristic of the base field plays a more central role in the definition than in any of the other
families mentioned so far. The Artin-Schreier construction is special because it cannot be obtained by base-change from a family of schemes over $\mathbb Z.$ Since Artin-Schreier covers are cyclic covers of $\mathbb{P}^1$, one obtains a direct link between their zeta functions and certain exponential sums; while this is also the case for cyclic $p$-fold covers in characteristics other than $p$, the Artin-Schreier case admits a much more precise analysis. One example of how to exploit this additional precision is the work of Rojas-Leon and Wan \cite{rw} refining the Weil bound for Artin-Schreier curves.

To explain our results in more detail, we introduce some notation.
Fix an odd prime $p$ and a finite field $\mathbb F_q$ of characteristic $p.$ Each polynomial $f \in \mathbb F_q[X]$ whose degree $d$ is not divisible by $p$ defines an Artin-Schreier cover $C_f$ of $\mathbb P^1$ with affine model
\begin{equation}\label{ascurve} Y^p-Y=f(X).\end{equation}
Since $f$ is a polynomial rather than a more general rational function, $C_f$ has $p$-rank $0.$ For more details about the structure of the moduli space of Artin-Schreier curves and its $p$-rank strata, see \cite{pz}. The Riemann-Hurwitz formula implies that the genus of the above curve is $g= (d-1)(p-1)/2.$ As usual, the Weil zeta function of $C_f$ has the form

\[Z_{C_f} (u) = \frac{P_{C_f}(u)}{(1-u)(1-qu)}.\] Here $P_{C_f}(u)$ is a polynomial of degree $2g=(d -1)(p-1)$ which factors as

\begin{equation}\label{product}P_{C_f}(u) = \prod_{\psi \neq 1} L(u, f, \psi),\end{equation}
where the product is taken over the non-trivial additive characters $\psi$ of $\mathbb F_p$ and $L(u,f,\psi)$ are certain $L$-functions (see \eqref{Euler-product} for the formula).
Computing the distribution of the zeroes of the zeta functions $Z_{C_f}(u)$ as $C_f$ runs over the $\mathbb{F}_q$-points of the moduli space $\mathcal{AS}_{g,0}$
of Artin-Schreier covers of genus $g$ and $p$-rank $0$ amounts to computing the distribution of  the zeroes
of $\prod_{j=1}^{p-1}L(u,f,\psi^j)$ for a fixed non-trivial additive character $\psi$ as $f$ runs over polynomials of degree $d$. In fact, going over each $\F_q$-point of the moduli space $\mathcal{AS}_{g,0}$ once is equivalent to letting $f$ vary over the set $\mathcal{F}'_d$ of polynomials of degree
$d$ containing no non-constant terms of degree divisible by $p$, as such terms can always be eliminated
in a unique way without changing the resulting Artin-Schreier cover.

Some statistics for the zeroes in the family of  Artin-Schreier covers were considered in the
recent work of Entin \cite{entin}, who employs the methods of Kurlberg and Rudnick \cite{kr} to study the variation of the number of points on
such a family, then translates the results into information about zeroes.
In the present work, we consider the global and mesoscopic regime, as  was done by Faifman and Rudnick \cite{fr} for the
family of hyperelliptic curves.

More precisely, we write \begin{equation}\label{factorization-of-L}
 L(u,f,\psi) = \prod_{j=1}^{d-1} (1-\alpha_{j}(f, \psi)u),
\end{equation}
\noindent where $\alpha_j(f, \psi) = \sqrt{q} e^{2 \pi i \theta_j(f, \psi)}$ and $\theta_j(f, \psi) \in [-1/2, 1/2)$.
We study the statistics of the set of angles $\{\theta_j(f, \psi)\}$ as $f$ varies. For an interval $\mathcal{I} \subset[-1/2,1/2),$ let
\[N_\mathcal{I}(f, \psi) := \#\{1\leq j \leq d-1:\,\theta_j(f, \psi) \in \mathcal{I}\},\]
\[
N_{\mathcal{I}}(f,\psi, \bar\psi):=N_\mathcal{I}(f, \psi)+ N_\mathcal{I}(f,\bar \psi),
\]
and 
\[
N_\mathcal{I}(C_f):=\sum_{j=1}^{p-1}N_\mathcal{I}(f,\psi^j).
\]
We show that the number of zeroes with angle in a prescribed non-trivial subinterval $\mathcal I$
is asymptotic to $2g|\mathcal I|$ (Theorem \ref{boundN_I}), has variance asymptotic to $\frac{2(p-1)}{\pi^2} \log(g|\mathcal I|)$ and properly normalized has a Gaussian distribution.

\begin{thm}\label{mainthm} Fix a finite field $\F_q$ of characteristic $p$. Let $\mathcal{F}_d'$ be the family of polynomials defined in \eqref{ourfamily}. Then for any real numbers $a<b$  and  $0<|\mathcal{I}|<1$ either  fixed or
$|\mathcal{I}|\rightarrow 0$ while $d |\mathcal{I}|\rightarrow \infty$,
\[\lim_{d \rightarrow \infty} \mathrm{Prob}_{\mathcal{F}_d'}\left(a < \frac{N_\mathcal{I}(C_f)-(d-1)(p-1)|\mathcal{I}|}{\sqrt{\frac{2(p-1)}{\pi^2}\log(d|\mathcal{I}|)}} < b\right)=\frac{1}{\sqrt{2\pi}} \int_a^b e^{-x^2/2} dx.\]
\end{thm}

As noted earlier, this result can also be stated in terms of the $\mathbb{F}_q$-points
of $\mathcal{AS}_{g,0}$. 

\begin{cor}\label{cor}  Fix a finite field $\F_q$ of characteristic $p$. Then for any real numbers $a<b$ and  $0<|\mathcal{I}|<1$ either  fixed or
$|\mathcal{I}|\rightarrow 0$ while $g |\mathcal{I}|\rightarrow \infty$,
\[\lim_{g \rightarrow \infty} \mathrm{Prob}_{\mathcal{AS}_{g,0}(\mathbb{F}_q)}\left(a < \frac{N_\mathcal{I}(C_f)-2g|\mathcal{I}|}{\sqrt{\frac{2(p-1)}{\pi^2}\log\left(g|\mathcal{I}|\right)}} < b\right)=\frac{1}{\sqrt{2\pi}} \int_a^b e^{-x^2/2} dx.\]
\end{cor}

Theorem \ref{mainthm} is obtained by computing the
normalized moments of certain approximations of 
\\${N_\mathcal{I}(C_f)-(p-1)(d-1)|\mathcal{I}|}$ given by Beurling-Selberg polynomials
to verify that they fit the Gaussian moments.
Our results are compatible with the following result for the distribution of zeroes of the $L$-functions $L(u,f, \psi)$ and $L(u, f, \bar \psi).$ 
\begin{prop}\label{prop}
 Fix a finite field $\F_q$ of characteristic $p$. Then for any real numbers $a<b$  and  $0<|\mathcal{I}|<1$ either  fixed or
$|\mathcal{I}|\rightarrow 0$ while $d |\mathcal{I}|\rightarrow \infty$,
 \[\lim_{d \rightarrow \infty} \mathrm{Prob}_{\mathcal{F}_d'}\left(a < \frac{N_{\mathcal{I}}(f,\psi, \bar\psi)-2(d-1)|\mathcal{I}|}{\sqrt{\frac{4}{\pi^2}\log(d|\mathcal{I}|)}} < b\right)=\frac{1}{\sqrt{2\pi}} \int_a^b e^{-x^2/2} dx.\]
\end{prop}

\begin{rem}
Analogous results hold for $N_{\mathcal I}(f,\psi)$ as long as the interval $\mathcal I$ is symmetric. 
\end{rem}

Notice that Proposition \ref{prop} is compatible with the philosophy of Katz and Sarnak, which predicts that when
$q \rightarrow \infty$, the distribution of $N_{\mathcal{I}}(C_f)$ is the same as the distribution of
$\hat{N}_{\mathcal{I}}(U)$, the number of eigenvalues of a $2g \times 2g$ matrix $U$ in the monodromy group
of $C_f$ chosen uniformly at random with respect to the Haar measure. The monodromy groups of Artin-Schreier covers are computed by Katz in
\cite{Katz1, Katz2}. In the large matrix limit, which corresponds to the limit as
$d \rightarrow \infty$ for the family of Artin-Schreier covers because $g = (p-1)(d-1)/2$, the statistics
on $\hat{N}_{\mathcal{I}}(U)$ have been found to have Gaussian fluctuations in various ensembles of
random matrices. 

\subsection{Outline of the article}
This article is set up as follows. We begin by reviewing basic Artin-Schreier theory in Section \ref{a-s}. In Section \ref{explicit} we prove two explicit formulas for the zeroes of $L(u, f, \psi)$ which we will need later to compute the moments. In Section \ref{distrzeros} we prove a result about the number of zeroes of the zeta function for a fixed Artin-Schreier cover of $\mathbb P^1$.  In Section \ref{trigapprox} we recall some facts on Beurling-Selberg polynomials and use them to prove some technical statements about their coefficients. A certain sum of these trigonometric polynomials approximate the characteristic function of the interval $\mathcal I$. 
We use the explicit formula to reduce the problem of studying this sum of Beurling-Selberg polynomials  to a problem about sums of characters of traces of a polynomial $f$ evaluated at elements in extensions of $\mathbb F_q$. In Sections \ref{1mom}, \ref{2mom} and \ref{3mom} we analyze the first, second and third 
moments of this sum. These moments tell us the expectation and
variance of the  distribution.  In Section \ref{genmom} we compute the general moments of our approximating function and conclude that it has a standard Gaussian limiting distribution as the degree $d$ of $f$ goes to infinity for $\mathcal I$ either fixed or in the mesoscopic regime. Finally, in Section \ref{proof} we conclude the proof of Theorem \ref{mainthm} by proving that under normalization $N_{\mathcal I}(C_f)-(d-1)(p-1)|\mathcal I|$ converges in mean square and hence distribution to our approximating function.

\section{Basic Artin-Schreier theory}\label{a-s}

We now recall some more facts about Artin-Schreier covers. For each integer $n\geq 1$, denote by $\tr_n: \mathbb F_{q^n} \to \mathbb F_p$ the absolute trace map (not the trace to $\F_q$). For each polynomial $g \in \mathbb F_q[X]$ and  non-trivial additive character $\psi$ of $\F_p$, set
\[
S_n(g, \psi) = \sum_{x\in \mathbb F_{q^n}} \psi(\tr_n(g(x))).
\]

The $L$-functions that appear in \eqref{product} are given by
\begin{equation}\label{Euler-product}
L(u,f,\psi) = \exp\left(\sum_{n=1}^{\infty} S_n(f, \psi) \frac{u^n}{n}\right) = \prod_P \left(1-\psi_f(P)u^{\deg P}\right)^{-1},
\end{equation}
where the product is taken over monic irreducible polynomials in $\F_q[X].$ In fact, throughout this paper $P$ will denote such a polynomial and, if $n=\deg P$ we have
\[\psi_f(P) = \sum_{\alpha \in \F_{q^n} \atop  P(\alpha)=0} \psi(f(\alpha))= \psi(\tr_n(f(\alpha))) \textrm{ for any root $\alpha$ of $P$.}\]
To see that the exponential is equal to the product over primes in \eqref{Euler-product}, one has to write the exponential as an Euler product over the closed points of $\mathbb A^1.$ Namely, if we denote by $\mathcal S_n$ the set of closed points of $\mathbb A^1$ of degree $n,$ we can write
\begin{eqnarray*}
L(u,f,\psi) &= & \exp\left(\sum_{n=1}^{\infty} S_n(f, \psi) \frac{u^n}{n}\right)\\
&=& \exp\left(\sum_{n=1}^\infty \sum_{x \in \mathcal S_n} \sum_{k=1}^\infty \psi(\tr_{kn}(f(x))) \frac{u^{kn}}{k} \right).
\end{eqnarray*}
 
 The denominator of the fraction is $k$,  not $kn,$ because each closed point $x \in \mathcal S_n$ produces $n$ rational points of $\F_{q^n}.$ Thus, 

\begin{eqnarray*}L(u,f,\psi) &=& \prod_{n=1}^\infty \prod_{x \in \mathcal S_n} \exp\left(\sum_{k=1}^\infty \frac{\big(\psi(\tr_{n}(f(x)))u^n\big)^k}{k} \right)\\
&=& \prod_{n=1}^\infty \prod_{x \in \mathcal S_n} (1-\psi(\tr_{n}(f(x)))u^n)^{-1}\\
&= &\prod_{x \textrm{ closed point of } \mathbb A^1} (1-\psi(\tr_{\deg x}(f(x)))u^{\deg x})^{-1},
\end{eqnarray*}
\noindent which is exactly the product over primes that appears in \eqref{Euler-product}.

Note that for the trivial character $\psi=1$, the same formula gives \[L(u,f,1)= Z_{\mathbb A^1}(u)= \frac{1}{1-qu}.\]
The factor at infinity is then given by \[\psi_f(\infty) = \begin{cases} 1 & \psi = 1, \\
0 & \psi \neq 1. \end{cases}\]
Therefore we have
\[Z_{C_f}(u) = \prod_{\psi} L^*(u,f,\psi),\]
where $L^*(u,f,\psi)$ are the completed $L$-functions,
\[L^*(u,f,\psi) = \prod_{v} \left(1-\psi_f(P_v) u^{\deg P_v}\right)^{-1}.\] Here the product is taken over all places $v$ of $\F_q(X).$

From now on we will fix a non-trivial additive character $\psi$ of $\mathbb F_p$ given by a certain choice $\zeta$ of a primitive $p$th root of unity in $\mathbb C.$ Then, all the other non-trivial characters of $\mathbb F_p$ are of the form $\sigma \circ \psi$ where $\sigma$ is an automorphism of the cyclotomic field $\mathbb Q(\zeta).$ The reciprocals of zeroes of the $L(u,f, \sigma\circ \psi)$ are exactly the Galois conjugates $\sigma(\alpha_j(f,\psi)),$ $1\leq j \leq d-1,$ of the reciprocals of the roots of $L(u, f, \psi).$ 
In order to compute the distribution of the zeroes of the Weil zeta functions $Z_{C_f}$ as $C_f$ runs over 
$\mathcal AS_{g,0}(\F_q)$ we are going to compute the distribution of the angles $\thet , \theta_j(f, \bar \psi), 1 \leq j \leq d-1,$ for our specific choice of the additive character $\psi,$ as $f$ runs through $\mathcal F_d',$ where $g=(d-1)(p-1)/2.$ 
Since the roots of $L(u, f, \psi)$ and $L(u, f, \bar{\psi})$ are conjugate, it suffices to work  with symmetric intervals. The distribution of the roots of the whole zeta function is then obtained by combining the $(p-1)/2$ distributions for the various choices of $\psi$.

As discussed in the introduction, we will consider $\F_q$-points of the moduli space $\mathcal{AS}_{g,0}$ of Artin-Schreier covers of $p$-rank $0.$ A cover consists of an Artin-Schreier curve for which we fix an automorphism of order $p$ and an isomorphism between the quotient and $\mathbb P^1.$ We also choose the ramification divisor to be $D=(\infty).$ Thus the one branch point of our $p$-rank $0$ covers is at infinity.

Concretely, we consider, up to $\F_q$-isomorphism, pairs of curves with affine model $C_f: Y^p - Y = f(X)$ with $f(X)$ a polynomial of degree $d = 2g/(p-1)+1$ not divisible by $p$ together with the automorphism $Y \mapsto Y+1.$  

Using  the $\F_q$-isomorphism $(X,Y) \mapsto (X, Y+aX^k)$, we get that $C_f$ is isomorphic to $C_g$ where
$g(X) = f(X) + aX^k - a^p X^{kp}$. By using this isomorphism, we are reduced to considering the Artin-Schreier curves with model
$C_f: Y^p-Y = f(X)$ where $f(X)$ is an element of the family  $\mathcal F_d'$  defined in the introduction as 
\begin{eqnarray} \label{ourfamily}
\mathcal F_d'=\left\{ a_dX^d + a_{d-1}X^{d-1} + \dots + a_0\in \F_q[X] : a_d \in \F_q^*, a_{pk}=0, 1 \leq k \leq \left\lfloor\frac{d}{p}\right\rfloor \right\}.
\end{eqnarray}
Except for the isomorphisms described above, no two such affine models are isomorphic.  Therefore considering all affine
 models $Y^p - Y = f(X)$ with $f(X) \in \mathcal F_d'$ is equivalent to considering
all the $\mathbb F_q$- points of the moduli space  $\mathcal{AS}_{g,0}.$ For more details on this one-to-one correspondence between our family and $\mathcal{AS}_{g,0} (\mathbb F_q),$ see \cite[Proposition 3.6]{pz}.

In \cite{entin}, the author is considering a slightly different family by
also allowing twists, i.e. isomorphism over
$\F_{q^p}$. This amounts to  the models $C_f: Y^p-Y = f(X),$ with $f(X) \in \mathcal F_d''$, where
\[
\mathcal F_d''=\left\{ a_dX^d + a_{d-1}X^{d-1} + \dots + a_0\in \F_q[X] : a_d \in \F_q^*, a_{pk}=0,
0 \leq k \leq\left\lfloor\frac{d}{p}\right\rfloor \right\}. \]

Finally, we will denote by
\begin{eqnarray*}
\mathcal F_d=\left\{ a_dX^d + a_{d-1}X^{d-1} + \dots + a_0\in \F_q[X] : a_d \in \F_q^* \right\},
\end{eqnarray*} \noindent the set of all polynomials of degree $d$ in $\F_q[X].$
We will also need the map $\mu: \mathcal{F}_d\rightarrow \mathcal{F}_d'$ defined by
\begin{equation}
\label{map-mu} \mu\left(\sum_{i=0}^d a_i X^i\right)= a_0+\sum_{{i=1}\atop{i \neq kp, k \geq 1}}^d \left( \sum_{j=0}^{\left\lfloor \log_p (d/i)\right\rfloor} a_{ip^j}^{p^{-j}}\right) X^i.
\end{equation} This map is $q^{\left\lfloor\frac{d}{p}\right\rfloor}$-to-one and preserves the
trace of $f(\alpha)$, which will allow us to work with  $\mathcal F_d$ instead of $\mathcal F'_d$ when taking averages.

\subsection{Remark on the number of points}

For $d$ large enough, the elements of $\mathcal F_d'$ have the same chance as any random polynomial of degree $d$ in $\mathbb F_q[X]$ to take a given value in some extension of $\mathbb F_q.$
Thus, if $p \nmid n,$ as soon as $d - \left\lfloor d/p \right\rfloor >q^n,$ the distribution of $\{\#C_f(\mathbb F_{q^n}): f \in \mathcal F_d'\}$ is given by a sum of i.i.d. random variables, one variable for each closed point of $\mathbb P^1$ of degree $e\mid n.$ As long as we stay away from the point at infinity where $f(X)$ has a pole, the fiber above each closed  point $x$ of $\mathbb P^1$ contains $pe$ rational points on the Artin-Schreier cover $C_f$ if $x$ happens to be in the kernel of the absolute trace map $\tr_{n}:\F_{q^n} \to \mathbb F_p,$ and no points otherwise. Hence each random variable in the sum takes the value $pe$ with probability $1/p$ and $0$ with probability $1-1/p.$ The average number of points is then $1+q^n,$ the constant $1$ coming from the point at infinity where the polynomial $f(X)$ has a pole and the fiber above it contains just $1$ point.

If $p\mid n,$ the average is higher because there are certain points of $\mathbb P^1$ of degree $e$ for which the fiber is forced to have $pe$ points (i.e. the points of degree $ e \mid \frac{n}{p}$). One adjusts the computation accordingly and obtains that the average number in $C_f(\F_{q^n})$ is now $1+q^n + (p-1)q^{n/p}.$  This is the essential reason behind Entin's result on the matter \cite[Theorem 4]{entin}, except that his count does not take into account the point at infinity.

\section{Explicit Formulas}\label{explicit}

Let $K$ be a positive integer,  $e(\theta) = e^{2\pi i \theta}$ and let  $h(\theta) = \sum_{{|k|}\leq K} a_k e(k\theta)$ be a trigonometric polynomial.
Then the coefficients $a_k$ are given by the Fourier transform
$$a_k = \widehat{h}(k) = \int_{-1/2}^{1/2} h(\theta) e (- k \theta) d\theta.$$

We prove in this section two explicit formulas for $L(u,f,\psi)$, written as an exponential of a sum or as a product
over primes as in (\ref{Euler-product}). The first explicit formula (Lemma \ref{Explicit-Formula}) will be used to compute
the moments over the family ${\mathcal{F}}_d'$, and the second explicit formula (Lemma \ref{Explicit-formula-relevant}) will be used to
prove a result about the number of zeroes for a fixed $C_f$  (see Section \ref{distrzeros}).

\begin{lem}\label{Explicit-Formula}
Let $h(\theta) = \sum_{{|k|}\leq K}\widehat{h}(k)e(k\theta)$  be a trigonometric polynomial. Let $\theta_j(f, \psi)$ be
the eigenangles of the $L$-function $L(u,f,\psi)$.
Then we have
\begin{equation}
\sum_{j=1}^{d-1}h(\thet) = (d-1)\widehat{h}(0) - \sum_{k=1}^{K}\frac{\widehat{h}(k)S_k(f,\psi) + \widehat{h}(-k)S_k(f,\overline{\psi})}{q^{k/2}}.
\end{equation}
\end{lem}
\begin{proof}
Recall from above that
$$L(u,f,\psi) = \exp\left(\sum_{n=1}^{\infty}S_n(f,\psi)\frac{u^n}{n}\right) = \prod_{j=1}^{d-1}(1- \alpha_j(f,\psi)u).$$
Taking logarithmic derivatives, we have
$$\frac{d}{du}\sum_{j=1}^{d-1}\log(1- \alpha_j(f,\psi)u) = \frac{d}{du}\sum_{n=1}^{\infty}S_n(f,\psi)\frac{u^n}{n}.$$
Multiplying both sides by $u,$ we get
$$\sum_{j=1}^{d-1}\frac{-\alpha_j(f,\psi)u}{1-\alpha_j(f,\psi)u} = \sum_{n=1}^{\infty}S_n(f,\psi)u^n,$$
that is,
$$-\sum_{j=1}^{d-1}\sum_{n=1}^{\infty}(\alpha_j(f,\psi)u)^n = \sum_{n=1}^{\infty}S_n(f,\psi)u^n.$$
Comparing coefficients,
$$-\sum_{j=1}^{d-1}(\alpha_j(f,\psi))^n = S_n(f,\psi).$$
Thus, for $n>0,$ we get
\begin{equation} \label{ngeq0} -\sum_{j=1}^{d-1}e^{2\pi i n \thet} = \frac{S_n(f,\psi)}{q^{n/2}}.\end{equation}
For $n<0,$ taking complex conjugates, we have  by (\ref{factorization-of-L}) and (\ref{ngeq0})
\begin{eqnarray*}
-\sum_{j=1}^{d-1}e^{2\pi i n \thet}
 &=& -\overline{\sum_{j=1}^{d-1}e^{2\pi i |n|\thet}}
 =-\overline{\sum_{j=1}^{d-1}\frac{\alpha_j(f, \psi)^{|n|}}{q^{|n|/2}}}\\
 &=&\overline{\frac{S_{|n|}(f,\psi)}{q^{|n|/2}}}
 = \frac{S_{|n|}(f,\overline{\psi})}{q^{|n|/2}} = \frac{S_{|n|}(f,\psi^{-1})}{q^{|n|/2}}.
 \end{eqnarray*}
Thus,
\begin{eqnarray*}
\sum_{j=0}^{d-1}h(\thet)\
&=&\sum_{j=1}^{d-1}\sum_{k=-K}^K\widehat{h}(k)e(k\thet)\\
&&=(d-1)\widehat{h}(0)+\sum_{j=1}^{d-1}\sum_{k=1}^K\widehat{h}(k)e(k\thet) +\sum_{j=1}^{d-1}\sum_{k = -K}^{-1}\widehat{h}(k)e(k\thet)\\
&&=(d-1)\widehat{h}(0) - \sum_{k=1}^K\widehat{h}(k)\left(\frac{S_k(f,\psi)}{q^{k/2}}\right) - \sum_{k = -K}^{-1}\widehat{h}(k)\left(\frac{S_{-k}(f,\overline{\psi})}{q^{-k/2}}\right)\\
&&=(d-1)\widehat{h}(0) - \sum_{k=1}^K\frac{\widehat{h}(k)S_k(f,\psi) + \widehat{h}(-k)S_k(f,\overline{\psi})}{q^{k/2}}.
\end{eqnarray*}
\end{proof}

\begin{lem}\label{Explicit-formula-relevant} Let $\theta_j(f, \psi)$ be
the eigenangles of the $L$-function $L(u,f,\psi)$.
Then for any $n\geq 1,$
$$ -\sum_{j=1}^{d-1}e^{2\pi i n \thet} = \sum_{\deg (M) = n}\frac{\Lambda(M)\psi_f(M)}{q^{n/2}}$$
where $M$ runs over monic polynomials in $\mathbb F_q[X],$
\[\Lambda(M) = \begin{cases}
\deg P&\textrm{ if } M = P^k\,\textrm{for some }k\geq 1\textrm{ and } P \textrm{ irreducible},\\ 0&\textrm { otherwise},
\end{cases}\]
and $\psi_f(P^k) = \psi_f(P)^k.$
\end{lem}

\begin{proof}
Comparing equations (\ref{Euler-product}) and (\ref{factorization-of-L}), we have
$$\prod_{j=1}^{d-1}(1 - \alp u) = \prod_P(1 - \psi_f(P)u^{\deg P})^{-1},$$
where the product on the right hand side is taken over monic irreducible polynomials in $\F_q[X].$
Taking logarithmic derivatives and multiplying by $u,$ we deduce that
$$-\sum_{j=1}^{d-1}\sum_{n=1}^{\infty}(\alp u )^n = \sum_M\Lambda(M)u^{\deg M}\psi_f(M).$$

Comparing the  coefficients of $u^n,$ we get
$$-\sum_{j=1}^{d-1}\alp^n = \sum_{\deg(M) = n}\Lambda(M)\psi_f(M),$$
and the result follows by dividing both sides by $q^{n/2}$.
\end{proof}

\section{The distribution of zeroes of $L(u,f,\psi)$} \label{distrzeros}

In this section we use the Erd\"{o}s-Tur\'{a}n inequality (see \cite{M}, Corollary 1.1) to prove a result on the number of  eigenangles $\theta_j(f, \psi)$ in an interval ${\mathcal{I}}$
for a fixed $L$-function $L(u,f,\psi)$.

\begin{thm}\label{Erdos-Turan}\text{[P. Erd\"{o}s, P. Tur\'{a}n]}
Let $x_1,x_2,\dots,x_N$ be real numbers lying in the unit interval $[-1/2,1/2).$  For any interval $\mathcal{I} \subseteq [-1/2,1/2),$
  let $A(\mathcal{I},N,\{x_n\})$ denote the number of
elements from the above set in $\mathcal{I}$.  Let $|\mathcal{I}|$ denote the length of the interval.  There exist absolute constants $B_1$ and $B_2$ such that for any $K\geq 1,$
$$|A(\mathcal{I},N,\{x_n\}) - N |\mathcal{I}|| \leq \frac{B_1N}{K+1} + B_2\sum_{k=1}^K\frac{1}{k}\left|\sum_{n=1}^{N}e^{2\pi i k x_n}\right|.$$
\end{thm}

We now prove the following theorem, which is the analogue of Proposition 5.1 in \cite{fr}.

\begin{thm}\label{boundN_I}
For any $\mathcal{I} \subseteq [-1/2,1/2),$ let $N_\mathcal{I}(f, \psi) := \#\{1\leq j \leq d-1:\,\thet \in \mathcal{I}\}.$  Then
$$N_\mathcal{I}(f,\psi) = (d-1) |\mathcal{I}| +O\left(\frac{d}{\log d}\right).$$
\end{thm}

\begin{proof}
 By the Erd\"{o}s-Tur\'{a}n inequality and Lemma \ref{Explicit-formula-relevant}, we have
\begin{eqnarray*}
|N_\mathcal{I}(f,\psi) - (d-1) |\mathcal{I}| |\
&\ll&\frac{d}{K} + \sum_{k=1}^K\frac{1}{k}\left|\sum_{\deg M = k}\frac{\Lambda(M)\psi_f(M)}{q^{k/2}}\right|\\
&&\ll \frac{d}{K} + \sum_{k=1}^K\frac{1}{q^{k/2}}\sum_{M = P^a,\,a\geq 1 \atop{\deg M = k}}1.
\end{eqnarray*}
Applying the function-field analogue of the prime number theorem, the above expression is
$\ll \displaystyle \frac{d}{K} + \frac{q^{K/2}}{K}.$ Choosing $K = \left[\frac{\log d}{\log q}\right],$ we deduce the theorem.
\end{proof}

\section{Beurling-Selberg functions} \label{trigapprox}

By the functional equation, the conjugate of a root of $Z_{C_f}(u)$ is also a root so we can restrict to considering symmetric intervals. Let  $0<\beta<1$ and set $\mathcal I = [-\beta/2, \beta/2] \subset [-1/2,1/2)$.  We are going to
approximate the characteristic function of ${\mathcal{I}}$, $\chi_{\mathcal{I}}$, with Beurling-Selberg polynomials $I_K^\pm$.
We will use the following properties of the coefficients of Beurling-Selberg polynomials (see \cite{M}, ch 1.2).

\begin{itemize}
 \item[{\bf (a)}] The $I^\pm_K$ are trigonometric polynomials of degree $\leq K$,
i.e., $$I_K^{\pm}(x) = \sum_{|k| \leq K} \widehat{I}_K^{\pm}(k) e(k x).$$
\item[{\bf (b)}] The Beurling-Selberg polynomials bound the characteristic function from below and above: \[I_K^- \leq \chi_{\mathcal{I}}\leq I_K^+.\]
\item[{\bf (c)}] The integral of Beurling-Selberg polynomials is close to the length of the interval: \[\int_{-1/2}^{1/2} I_K^\pm(x) dx =\int_{-1/2}^{1/2} \chi_{\mathcal{I}}(x) dx \pm \frac{1}{K+1}.\]
\item[{\bf (d)}] The $I^\pm_K$ are even (since we are taking the interval $\mathcal{I}$
to be symmetric about the origin). It then follows that the Fourier coefficients
are also even, i.e. $\widehat{I}_K^{\pm}(-k) = \widehat{I}_K^{\pm}(k)$ for
$|k| \leq K$.
\item[{\bf (e)}] The nonzero Fourier coefficients are also close to those of the characteristic function:
\[|\widehat{I}_K^\pm (k) - \widehat{\chi}_{
 \mathcal{I}}(k) | \leq \frac{1}{K+1} 
\quad \Longrightarrow \quad \widehat{I}^\pm_K(k)=\frac{\sin (\pi k|\mathcal{I}|)}{\pi k} + O
\left( \frac{1}{K+1}\right), \quad k \geq 1.
\]
This implies the following bound:
\[|\widehat{I}_K^\pm (k)| \leq \frac{1}{K+1} +\min \left \{|\mathcal{I}|, \frac{\pi}{|k|}\right \}, \quad 0<|k|\leq K;\]
\end{itemize}

\begin{prop}(Proposition 4.1, \cite{fr}) \label{propFR} For $K\geq 1$ such that $K|\mathcal{I}|>1$, we have
\begin{eqnarray*}
\sum_{k \geq 1} \widehat{I}_K^\pm (2k)&=&O(1),\\
\sum_{k \geq 1} \widehat{I}_K^\pm (k)^2 k&=&\frac{1}{2\pi^2} \log (K|\mathcal{I}|) +O(1),\\
\sum_{k \geq 1} \widehat{I}_K^+ (k)\widehat{I}_K^- (k) k&=&\frac{1}{2\pi^2} \log (K|\mathcal{I}|) +O(1).\\
\end{eqnarray*}
\end{prop}
Note that for a given $K$ these sums are actually finite, since the Beurling-Selberg polynomials $I_K^\pm$ have degree at most $K$. 
\begin{proof} The first two statements are proven in Proposition 4.1 of \cite{fr}.
Since $$\widehat{I}_K^{\pm}(k) = \frac{\sin (\pi k|\mathcal{I}|)}{\pi k} + O
\left( \frac{1}{K}\right),$$
holds for both $\widehat{I}_K^{+}(k)$ and $\widehat{I}_K^{-}(k),$ the  third statement follows by exactly the same proof as the second statement.
\end{proof}

We will also need the following estimates.
\begin{prop} \label{propmanysums} For $\alpha_1,\dots, \alpha_r, \gamma_1,\dots, \gamma_r>0$, and $\beta_1,\dots,\beta_r \in \mathbb{R}$, we have,
 \[\sum_{k_1,\dots,k_r \geq 1} {\widehat{I}_K^\pm(k_1)}^{\alpha_1} \dots {\widehat{I}_K^\pm(k_r)}^{\alpha_r}k_1^{\beta_1}\dots k_r^{\beta_r}q^{-\gamma_1k_1-\dots-\gamma_r k_r}=O(1).\]
For $\alpha_1,\alpha_2,\gamma >0$, and $\beta \in \mathbb{R}$,
\[\sum_{k\geq1} {\widehat{I}_K^\pm(k)}^{\alpha_1}  {\widehat{I}_K^\pm(2k)}^{\alpha_2} k^\beta q^{-\gamma k}=O(1).\]
\end{prop}

\begin{proof} Since $\left|\widehat{I}_K^\pm(k)\right|\leq \frac{1}{K+1}+\min\left\{|\mathcal{I}|, \frac{\pi}{|k|}\right\}$, we obtain
\begin{eqnarray*}
&&\left|\sum_{k_1,\dots,k_r \geq 1} {\widehat{I}_K^\pm(k_1)}^{\alpha_1} \dots {\widehat{I}_K^\pm(k_r)}^{\alpha_r}k_1^{\beta_1}\dots k_r^{\beta_r}q^{-\gamma_1k_1-\dots-\gamma_r k_r}\right|
\ll \sum_{k_1,\dots,k_r\geq 1} k_1^{\beta_1}\dots k_r^{\beta_r}q^{-\gamma_1k_1-\dots-\gamma_r k_r}
\end{eqnarray*}
Since $\sum_{k\geq 1} k^\beta q^{-\gamma k}=O(1)$ for $q>1$ and $\gamma>0$, we get that the right hand side above is also equal to $O(1).$
The second equation is a particular form of the more general equation established above.
\end{proof}

\section{First Moment} \label{1mom}

Recall that $N_{\mathcal I} (f,\psi)$ denotes the number of angles $\thet$ of the zeroes of the $L$-function $L(u,f,\psi)$ in the interval $\mathcal I \subset [-1/2, 1/2)$ of length $0<|\mathcal I |<1.$

From now on, for a function $\phi:\mathcal{F}_d' \rightarrow\mathbb C$, we denote its average by  \[\left<\phi(f)\right>:=\frac{1}{|\mathcal{F}_d'|}
\sum_{f \in \mathcal{F}_d'}\phi(f).\]

We want to compute the first moment
\begin{eqnarray*}
\left< N_{\mathcal I}(f,\psi)  \right> = \frac{1}{|\mathcal{F}_d'|}
\sum_{f \in \mathcal{F}_d'} N_{\mathcal I}(f,\psi) .
\end{eqnarray*}
We will do so by proving the following result.
\begin{thm}  \label{averageET} As $d \rightarrow \infty$,
\begin{eqnarray*}
\left< N_{\mathcal I}(f,\psi) - (d-1) |\mathcal{I}| \right> = O(1).
\end{eqnarray*}
\end{thm}

\begin{rem} Recall that in Theorem \ref{boundN_I} we showed that
$$N_{\mathcal I}(f,\psi) - (d-1) |\mathcal{I}| = O\left(\frac{d}{\log d}\right).$$ Theorem \ref{averageET}, on the other hand, gives us a far better estimate for the average of $\left< N_{\mathcal I}(f,\psi) - (d-1) |\mathcal{I}| \right>$ than we could have derived from Theorem \ref{boundN_I}.
\end{rem}

For the proof of Theorem \ref{averageET}, we will use the Beurling-Selberg approximation of the characteristic function of the interval $\mathcal I.$
By property {\bf (b)} of the Beurling-Selberg polynomials,
$$
\sum_{j=1}^{d-1} I_K^{-}(\theta_j(f,\psi))  \leq N_\mathcal{I}(f, \psi) \leq \sum_{j=1}^{d-1} I_K^{+}(\theta_j(f,\psi)) .
$$

With the explicit formula of Lemma \ref{Explicit-Formula} and property {\bf (c)},
we write
\begin{eqnarray*}
\sum_{j=1}^{d-1} I_K^{\pm}(\theta_j(f,\psi))  &=& (d-1) |\mathcal{I}| - S^{\pm}(K, f, \psi) \pm \frac{d-1}{K+1} \\
\end{eqnarray*}
where
\begin{eqnarray} \label{defSK}
S^{\pm}(K, f,\psi) :=  \sum_{k=1}^K \frac{\widehat{I}^\pm_K(k)S_k(f, \psi)+\widehat{I}^\pm_K(-k)S_k(f, \bar{\psi})}{q^{k/2}} . \end{eqnarray}

This gives
\begin{eqnarray} \label{T-estimate}
- S^{-}(K, f,\psi) -\frac{d-1}{K+1}  \leq N_\mathcal{I}(f, \psi) - (d-1)|\mathcal{I}| \leq - S^{+}(K, f,\psi) + \frac{d-1}{K+1}.
\end{eqnarray}

In order to complete the proof it remains to estimate $\langle S^{\pm}(K, f, \psi) \rangle.$ We will need the following results from \cite{entin}.
As we remarked in Section \ref{a-s},
we are using a slightly different description for the family of Artin-Schreier covers since we do not allow twists.
Because of that, our results are slightly simpler than those stated in \cite{entin}. We have also modified the original notation so that it fits the generalization that we pursue in the next sections.

\begin{lem} (\cite{entin}, Lemma 5.2)\label{lem:avg}Let $h$ be an integer, $p\nmid h$.  Assume $k<d$ and $\alpha \in \F_{q^k}.$ Then
\[\left< h \psi(\tr_k f(\alpha))\right>  = \begin{cases} 1 & p\mid k, \, \alpha \in \F_{q^{k/p}},\\
0 & \textrm{otherwise.}
\end{cases} \]
\end{lem}
\begin{proof}
If $p \mid k$ and $\alpha \in \F_{q^{k/p}}$ then $\tr_k(f(\alpha))=p \tr_{\frac{k}{p}}(f(\alpha))=0$
so $\left< \psi(\tr_k f(\alpha))\right>=1$. For the remaining case we first note that the average is the same if
we average over the family $\mathcal F_d$ of degree $d$ polynomials (without the condition $a_{pk}=0$).
This is due to the existence of the map $\mu$ defined by (\ref{map-mu}).

Denote by $u$ the degree of $\alpha$ over $\F_q$. Since $u\leq k<d$ the  map
\[
\tau : \mathcal F_d\rightarrow \F_{q^{u}}
\]
defined by $\tau(f)=f(\alpha)$ is $(q-1)q^{d-u}$-to-one. Thus as $f$ ranges over $\mathcal F_d$, $f(\alpha)$ takes each value
in $\F_{q^u}$ an equal number of times. Since $p\nmid \frac{k}{u}$, $\tr_k(f(\alpha))=\frac{k}{u}\tr_u(f(\alpha))$ also takes every value in $\F_p$ the same number of times as $f$ ranges over $\mathcal F_d$ and the same is true for $h\tr_k(f(\alpha))$. Thus each $p$th root of unity occurs the same number of times in $\psi(h\tr_k(f(\alpha)))$ as $f$ ranges over $\mathcal F_d$ and so the average is $0$.
\end{proof}

The lemma has the following consequence.
 \begin{cor} (\cite{entin}, Corollary 5.3) \label{cor:moment1}
Let $h$ be an integer, $p\nmid h$.   Assume $k <d$ and set
\[M^{k,1,h}_{1,d}:= \left<q^{-k/2} \sum_{\alpha \in \F_{q^k}} \psi(h \tr_k f(\alpha))\right>.\] Then \[M^{k,1,h}_{1,d}= e_{p,k}q^{-(1/2-1/p)k},\]
where
\[e_{p,k}=\begin{cases} 0 & p\nmid k, \\
1 & p\mid k. \end{cases}\]
\end{cor}
We also denote
\[M^{k,-1,h}_{1,d}:= \left<q^{-k/2} \sum_{\alpha \in \F_{q^k}}  \psi(- h \tr_k f(\alpha))\right>.\]
Clearly, $M^{k,-1,h}_{1,d}=\overline{M^{k,1,h}_{1,d}}.$

Notice that changing $h$ allows us to vary the character from $\psi$ to $\psi^h$. This will be useful later. 

\begin{proof}(Theorem \ref{averageET})
We have that
\begin{eqnarray*}
\left<  S^{\pm}(K, f, \psi) \right> &=&  \sum_{k=1}^K \frac{\widehat{I}^\pm_K(k)\left<S_k(f,\psi) \right>+\widehat{I}^\pm_K(-k)\left<S_k(f,\bar{\psi})\right>}{q^{k/2}}\\
&=&  \sum_{k=1}^K \widehat{I}^\pm_K(k)M^{k,1,1}_{1,d}+\widehat{I}^\pm_K(-k)M^{k,-1,1}_{1,d}\\
&= &  2\sum_{k=1}^K \widehat{I}^\pm_K(k)e_{p,k}q^{-(1/2-1/p)k}\\
\end{eqnarray*}
and the result follows from property {\bf (e)} and \eqref{T-estimate} taking $K=cd$ with $c<1$.
\end{proof}

\begin{rem} \label{Chantal's favorite nonconstant} We denote by
\[C(K):=\sum_{k=1}^K \widehat{I}^\pm_K(k)e_{p,k}q^{-(1/2-1/p)k}\]
and
\[
C:=\sum_{k=1}^\infty \frac{\sin(\pi k |\mathcal I|)}{\pi k} e_{p,k} q^{-(1/2-1/p)k}.
\]
These terms will reappear in the computation of the higher moments. Note that, since $p>2,$ the above infinite series converges absolutely. By Proposition \ref{propmanysums}, $C(K)=O(1)$. By property {\bf (e)} of the Beurling-Selberg polynomials, $C=C(K) + O(1/K)$.
\end{rem}

\section{Second moment} \label{2mom}
Let
\begin{eqnarray}
S^\pm(K,C_f)=\sum_{h=1}^{p-1}S^{\pm}(K, f, \psi^h),
\end{eqnarray}
where $S^{\pm}(K, f, \psi)$ is defined in \eqref{defSK}.

In the next sections, we are computing the moments of $S^{\pm}(K, C_f)$. We show that they fit the Gaussian moments when properly normalized
(Theorem \ref{thm:sumisgaussian}). We will then use this result to show that 
$$ \frac{N_{\mathcal{I}}(C_f)
- (p-1)(d-1) |\mathcal{I}|}{\sqrt{\frac{2(p-1)}{\pi^2} \log(d |\mathcal{I}|)}}$$
converges to a normal distribution as $d \rightarrow \infty$ since
it converges in mean square to $$\frac{S^{\pm}(K, C_f)}{{\sqrt{\frac{2(p-1)}{\pi^2} \log(d |\mathcal{I}|)}}}.$$

The following lemma is a generalization of Lemma 6.2 in \cite{entin}, that also takes into account the difference in our family of Artin-Schreier covers.

Recall that $\psi^j(\alpha)=\psi(j\alpha)$ for $\alpha \in \F_p$.  We have the following 

\begin{lem} \label{lemmaindependence}
Fix $h_1, h_2$ such that $p\nmid h_1 h_2$ and let $e_1, e_2 \in \{-1,1 \}$. Assume $k_1,k_2 > 0$, $k_1+k_2 < d$. Let $\alpha_1 \in \F_{q^{k_1}}$, $\alpha_2 \in \F_{q^{k_2}}$
with monic minimal polynomials $g_1,g_2$ of degrees $u_1, u_2$ over $\F_q$ respectively.
 We have
\begin{eqnarray*}
\left< \psi(e_1 h_1\tr_{k_1} f(\alpha_1)+e_2h_2\tr_{k_2} f(\alpha_2)) \right> &=&\begin{cases}
1, & \textrm{$g_1=g_2$, $p \mid \frac{{e_1 h_1 k_1}+{e_2 h_2 k_2}}{u_1}$, $p \nmid \frac{k_1k_2}{u_1u_2}$} \\ 
& \textrm{or $p\mid \left(\frac{k_1}{u_1}, \frac{k_2}{u_2}\right)$}; \\ 
0, & \textrm{otherwise.} \end{cases}
\end{eqnarray*}
\end{lem}
\begin{proof}
If $p \mid \frac{k_2}{u_2}$ then $\tr_{k_2} f(\alpha_2)=p\tr_{{\frac{k_2}{p}}}f(\alpha_2)=0$, so
\[
\left< \psi(e_1 h_1\tr_{k_1} f(\alpha_1) +e_2 h_2 \tr_{k_2} f(\alpha_2)) \right> =\left< \psi( e_1 h_1 \tr_{k_1} f(\alpha_1)) \right>.
\]
By Lemma \ref{lem:avg}, this equals $0$ if $p \nmid \frac{k_1}{u_1}$ and $1$ if $p\mid \frac{k_1}{u_1}$ as $p \nmid e_1h_1$.

The only remaining case is when $p \nmid \frac{k_1k_2}{u_1u_2}$. We first suppose that $g_1 \neq g_2$.
We note that we will have the same value if we average over $\mathcal F_d$ rather than $\mathcal F_d'$  due to the existence of the map $\mu$ defined by (\ref{map-mu}). Since $u_1+u_2\leq {k_1}+{k_2} <d$, the map
\[
\tau : \mathcal F_d \rightarrow \F_q[X]/(g_1g_2) \simeq \F_{q^{u_1}}\times \F_{q^{u_2}}
\]
is exactly $(q-1)q^{d-{u_1-u_2}}$-to-one. Hence as $f$ ranges over $\mathcal F_d$, $(f(\alpha_1), f(\alpha_2))$ takes every value in $\F_{q^{u_1}}\times \F_{q^{u_2}}$ the same number of times. Now, since $p\nmid  \frac{{e_1 h_1 k_1}}{u_1}$ and $p \nmid \frac{{e_2 h_2 k_2}}{u_2}$, 
\[(\tr_{k_1} f(\alpha_1), \tr_{k_2} f(\alpha_2))=\left(\frac{{e_1h_1k_1}}{u_1}\tr_{u_1}(f(\alpha_1)), \frac{{e_2h_2k_2}}{u_2}\tr_{u_2}(f(\alpha_2))\right)\] also takes every value in $\F_p\times \F_p$
 the same number of times as $f$ ranges over $\mathcal F_d$. Then
\begin{eqnarray*}
&&\psi\left(e_1h_1\tr_{{k_1}}(f(\alpha_1))+e_2h_2  \tr_{{k_2}} (f(\alpha_2))\right)= \\&=&\psi\left(e_1h_1\frac{{k_1}}{u_1}\tr_{u_1}(f(\alpha_1))+e_2h_2\frac{{k_2}}{u_2} \tr_{u_2} (f(\alpha_2))\right)
\end{eqnarray*}
assumes every $p$th root of unity equally many times as we average over $\mathcal F_d$ and so the average is $0$.

If $g_1=g_2$, then  $\alpha_1$ and $\alpha_2$ are conjugates over $\F_q$ and so are $f(\alpha_1)$ and  $f(\alpha_2).$ Then
$\tr_{u_1} f(\alpha_1)=\tr_{u_1} f(\alpha_2)$. This implies 
\begin{eqnarray*}
e_1h_1\tr_{k_1} f(\alpha_1)+e_2h_2\tr_{k_2} f(\alpha_2) &=& e_1h_1\frac{{k_1}}{u_1}\tr_{u_1} f(\alpha_1) +e_2h_2 \frac{{k_2}}{u_1}\tr_{u_1} f(\alpha_2) = \\ &=& \frac{{e_1h_1k_1}+e_2h_2 {k_2}}{u_1} \tr_{u_1} f(\alpha_1),
\end{eqnarray*} which is
zero when $p\mid \frac{e_1h_1{k_1}+e_2h_2{k_2}}{u_1}$.  If $p$ does not divide $\frac{e_1h_1{k_1}+e_2h_2{k_2}}{u_1}$ then
\[\left< \psi(e_1h_1\tr_{k_1} f(\alpha_1)+e_2h_2 \tr_{k_2} f(\alpha_2)) \right> =
\left< \psi\left(\frac{{e_1h_1k_1}+e_2h_2 {k_2}}{u_1}  \tr_{u_1} f(\alpha_1)\right) \right> =0\] by Lemma \ref{lem:avg}.

\end{proof}

For positive integers $k_1,k_2,h_1,h_2$ with $p\nmid h_1h_2$ and $e_1, e_2 \in \{ -1, 1\}$, let
\begin{eqnarray*}
M_{2,d}^{(k_1,k_2),(e_1,e_2),(h_1,h_2)} &:=& \left<  q^{-(k_1+k_2)/2} \sum_{{\alpha_1 \in \F_{q^{k_1}}}\atop{\alpha_2 \in \F_{q^{k_2}}}}\psi(e_1h_1\tr_{k_1} f(\alpha_1)+e_2h_2 \tr_{k_2} f(\alpha_2)) \right> \\
&=& q^{-(k_1+k_2)/2} \sum_{{\alpha_1 \in \F_{q^{k_1}}}\atop{\alpha_2 \in \F_{q^{k_2}}}}
\left< \psi(e_1h_1\tr_{k_1} f(\alpha_1) +e_2h_2 \tr_{k_2} f(\alpha_2)) \right> .
\end{eqnarray*}
Then we have the following analogue of Theorem 8 in \cite{entin}.
\begin{thm}\label{Mcovariance} Assume ${k_1} \geq {k_2} > 0$ and ${k_1}+{k_2} < d$. Let $0<h_1, h_2 \leq (p-1)/2$.
Then
\begin{eqnarray*}
M_{2,d}^{({k_1},{k_2}),(e_1,e_2),(h_1,h_2)}
&=&\delta_{{k_1},2{k_2}} O \left({k_1} q^{-{k_2}/2}\right) + O \left({k_1} q^{-{k_2}/2-{k_1}/6} + q^{-(1/2 - 1/p)({k_1}+{k_2})}\right)\\
&&+
\begin{cases}
 \delta_{{k_1},{k_2}} {k_1} \left(1+O\left(q^{-{k_1}/2}\right)\right), & (e_1,e_2)=(1,-1), h_1=h_2,\\
 0, & \text{ otherwise,}
\end{cases}
\end{eqnarray*}
where
\begin{eqnarray*}
\delta_{{k_1},{k_2}} = \begin{cases} 1, & {k_1}={k_2}, \\0, & {k_1} \neq {k_2}. \end{cases}
\end{eqnarray*}
\end{thm}
Before  we proceed with the proof, we would like to make a few remarks. In the instances when we apply this result,  we will choose $K=cd$, for $0<c<1/2$, and therefore ${k_1}, {k_2}\leq K$ will imply
that $k_1 + k_2 < d$, and will be able to
apply Theorem \ref{Mcovariance} for all values of $k_1, k_2$ under consideration.
Also note that the condition $k_1\geq k_2>0$ does not restrict the validity of the statement, since
$M_{2,d}^{({k_2},{k_1}),(1,-1),(h_1,h_2)}=\overline{M_{2,d}^{({k_1},{k_2}),(1,-1), (h_2,h_1)}}$.

\begin{proof} From Lemma \ref{lemmaindependence},
\begin{eqnarray*}
&&M_{2,d}^{({k_1},{k_2}),(e_1,e_2),(h_1,h_2)}  = q^{-({k_1}+{k_2})/2}\left(e_{p,e_1h_1{k_1}+e_2h_2{k_2}}\sum_{{{m\mid ({k_1},{k_2})}\atop{mp\nmid {k_1},{k_2}}}\atop{mp \mid (e_1h_1{k_1}+e_2 h_2{k_2})}} \pi(m)m^2 +e_{p,{k_1}}e_{p,{k_2}}q^{({k_1}+{k_2})/p}\right),
\end{eqnarray*}
where $\pi(m)$ denotes the number of monic irreducible polynomials of degree $m$ over $\F_q[X]$. The prime number theorem for function fields (see \cite{rosen}, Theorem 2.2) states that $\pi(m) =\frac{q^m}{m}+O\left(\frac{q^{m/2}}{m}\right).$

When ${k_1}={k_2}$,  the conditions on the summation indices become $m\mid {k_1}$, $mp\nmid {k_1}$, and $mp\mid (e_1h_1+e_2h_2){k_1}$, a contradiction unless $p\mid (e_1h_1+e_2h_2)$. Due to the range in which the $h_1, h_2$ take values, this can only happen when $e_1=-e_2$ and $h_1=h_2$. 
In this case, one gets
\[\sum_{{m\mid {k_1}}\atop{mp\nmid {k_1}}} \pi(m)m^2 = {k_1}q^{k_1}+O\left({k_1}q^{{k_1}/2}\right).\]
On the other hand, when ${k_1}=2{k_2}$, one gets
\[\sum_{{{m\mid {k_2}}\atop{mp\nmid {k_2}}}\atop{mp \mid (2e_1h_1+e_2h_2){k_2}}} \pi(m)m^2=O({k_2}q^{k_2})=O\left({k_1}q^{{k_1}/2}\right).\]
Finally, if ${k_1}>{k_2}$ but ${k_1}\not = 2{k_2}$, we have $({k_1},{k_2})\leq {k_1}/3$ and
\[\sum_{{{m\mid ({k_1},{k_2})}\atop{mp\nmid {k_1},{k_2}}}\atop{mp \mid (e_1h_1{k_1}+e_2h_2{k_2})}} \pi(m)m^2=O\left({k_1} q^{{k_1}/3}\right).\]
This concludes the proof of the theorem.
\end{proof}

Finally, we are able to compute the covariances.
\begin{thm}\label{covariance}Let  $h_1, h_2$ be integers such that $0< h_1, h_2\leq (p-1)/2$. Then for any $K$ with  $\max\{1,1/|\mathcal I|\}<K<d/2$, 
\begin{eqnarray*}
 \left< S^\pm(K, f, \psi^{h_1}) S^\pm(K, f, \psi^{h_2})\right> &=&\left< S^\pm(K, f, \psi^{h_1}) S^\mp(K, f, \psi^{h_2}) \right>
  =\begin{cases}

\displaystyle \frac{1}{\pi^2} \log (K |\mathcal{I}|)+ O\left(1\right), & h_1=h_2
 \\
 &\\
  O\left(1\right), & h_1\neq h_2.
  \end{cases}
\end{eqnarray*}

\end{thm}
\begin{proof}
 By definition,
\begin{eqnarray*}
 &&\left< S^\pm(K, f, \psi^{h_1}) S^\pm(K, f, \psi^{h_2}) \right>  \\
&&= \sum_{{k_1},{k_2} =1}^K \widehat{I}_K^\pm({k_1}) \widehat{I}_K^\pm({k_2}) M_{2,d}^{({k_1},{k_2}),(1,1),(h_1,h_2)}  +
\widehat{I}_K^\pm({k_1}) \widehat{I}_K^\pm(-{k_2}) M_{2,d}^{({k_1},{k_2}),(1,-1),(h_1,h_2)}  \\
&& \;\;\;\; + \widehat{I}_K^\pm(-{k_1}) \widehat{I}_K^\pm({k_2}) M_{2,d}^{({k_1},{k_2}),(-1,1),(h_1,h_2)}+\widehat{I}_K^\pm(-{k_1}) \widehat{I}_K^\pm(-{k_2}) M_{2,d}^{({k_1},{k_2}),(-1,-1),(h_1,h_2)}.  
\end{eqnarray*}
Then, by repeated use of Theorem \ref{Mcovariance} 
and
Proposition \ref{propmanysums}, the summation over $k_1, k_2$ is $O(1)$ if $h_1\neq h_2$. If $h_1=h_2$ then
\begin{eqnarray*}
\left< S^\pm(K, f, \psi^{h_1}) ^2 \right> &=& 2 \sum_{{k_1}=1}^K \widehat{I}_K^\pm({k_1})\widehat{I}_K^\pm({-k_1})  {k_1} + O(1)
= 2 \sum_{{k_1}\geq 1} \widehat{I}_K^\pm({k_1})^2 {k_1} 
 + O ( 1 )  \\
&=& \frac{1}{\pi^2} \log(K |\mathcal{I}|) + O(1)
\end{eqnarray*}
by applying Proposition \ref{propFR}.
The proof for  $\left< S^\pm(K, f, \psi^{h_1}) S^\mp(K, f, \psi^{h_2}) \right>$ follows along exactly the same lines.
\end{proof}

\begin{cor}\label{cor:2ndmoment}
For any  $K$ with $\max\{1,1/|\mathcal I|\}<K<d/2$, 
\[
\langle S^\pm(K,C_f)^2 \rangle=\langle S^+(K,C_f)S^-(K,C_f)\rangle=\frac{2(p-1)}{\pi^2}\log(K|\mathcal I|) + O(1).
\]
\end{cor}
\begin{proof}
First we note that
\begin{eqnarray*}
\langle S^\pm(K,C_f)^2\rangle=  \sum_{h_1,h_2=1}^{p-1} \left\langle S^{\pm}(K, f, \psi^{h_1}) S^{\pm}(K, f, \psi^{h_2})\right\rangle.
\end{eqnarray*}
Notice that by Theorem \ref{covariance}, the mixed average contributes $\frac{1}{\pi^2}\log(K|\mathcal I|)+O(1)$ for each term where $h_1=h_2$ or $h_1=p-h_2$. The proof for $\langle S^+(K,C_f)S^-(K,C_f)\rangle$ is identical. 
\end{proof}

\section{Third moment} \label{3mom}

Let $k_1, k_2, k_3$ be positive integers, $e_1, e_2, e_3$ take values $\pm 1$, and $h_1,h_2,h_3$ be integers such that $p\nmid h_i$. Denote ${\bf k}=(k_1,k_2,k_3)$, ${\bf e}=(e_1,e_2,e_3)$, and ${\bf h}=(h_1,h_2,h_3)$.
For every ${\bm \alpha}=(\alpha_1,\alpha_2,\alpha_3) \in \F_{q^{k_1}} \times \F_{q^{k_2}} \times \F_{q^{k_3}} $, set
$$
m_{3,d}^{{\bf k}, {\bf e}, {\bf h}}({\bm \alpha}) =
\left< \psi(e_1 h_1 \tr_{k_1}f(\alpha_1) + e_2 h_2 \tr_{k_2}f(\alpha_2)  +
e_3 h_3 \tr_{k_{3}} f(\alpha_3)) \right>,
$$
and
$$M_{3, d}^{{\bf k}, {\bf e}, {\bf h}}=
 \sum_{{\alpha_i \in \F_{q^{k_i}}} \atop {i=1,2,3}}
q^{-(k_1 + k_2 + k_3)/2} m_{3,d}^{{\bf k}, {\bf e}, {\bf h}}({\bm \alpha}).$$

In an analogous manner to Section \ref{2mom}, one can prove the following.
\begin{lem}\label{lemm3}
Let $p\nmid h_1h_2h_3$ and let $e_1, e_2, e_3 \in \{-1, 1\}$. Assume $k_1, k_2, k_3 > 0$ and $k_1+k_2+k_3 < d$. For $i=1,2,3$  $\alpha_i$ be an element of  $\F_{q^{k_i}}$
with minimal polynomial $g_i$ over $\F_q$  of degree $u_i.$
We have $m_{3,d}^{{\bf k}, {\bf e}, {\bf h}}({\bm \alpha})=1$ in any of the following cases
\begin{itemize}
 \item $g_1=g_2=g_3, p  \mid \frac{(e_1h_1k_1+e_2h_2k_2+e_3h_3k_3)}{u_1}, p\nmid\frac {k_1k_2k_3}{u_1u_2u_3}$.
 \item $g_{j_1}=g_{j_2}, p \mid \frac{(e_{j_1}h_{j_1}k_{j_1}+e_{j_2}h_{j_2}k_{j_2})}{ u_{j_1}}, p \nmid \frac{k_{j_1}k_{j_2}}{u_{j_1}u_{j_2}}, p \mid\frac{k_{j_3}}{u_{j_3}}$ , where $(j_1,j_2,j_3)$ is any permutation of $(1,2,3)$.
\item $p \mid \frac{k_i}{u_i}, i=1,2,3$.
\end{itemize}
Otherwise $m_{3,d}^{{\bf k}, {\bf e}, {\bf h}}({\bm \alpha})=0$.

\end{lem}

\begin{thm}\label{thm3} Assume $k_1 \geq k_2 \geq k_3 > 0$ and $k_1 + k_2 + k_3 < d$. Then
\begin{eqnarray*}
&&M_{3, d}^{{\bf k}, {\bf e}, {\bf h}}\\ &=&M_{1,d}^{k_1,e_1, h_1}M_{2,d}^{(k_2,k_3),(e_2,e_3), (h_1,h_2)}+M_{1,d}^{k_2,e_2, h_3}M_{2, d}^{(k_1,k_3),(e_1,e_3), (h_1,h_3)}\\ &&+M_{1,d}^{k_3,e_3, h_3}M_{2, d}^{(k_1,k_2),(e_1,e_2),(h_1,h_2)} -2M_{1,d}^{k_1,e_1,h_1}M_{1,d}^{k_2,e_2,h_2}M_{1,d}^{k_3,e_3,h_3}\\
&&+O\left(\delta_{k_1,k_2,k_3} k_1^2q^{-k_1/2}+\delta_{k_1,k_2,2k_3} k_1^2 q^{-3k_1/4} +  \delta_{k_1,2k_2,2k_3} k_1^2 q^{-k_1/2}+k_1^2 q^{-k_1/6-k_2-k_3} \right)\\
&=& e_{p,k_1}q^{-(1/2-1/p)k_1}M_{2,d}^{(k_2,k_3),(e_2,e_3),(h_2,h_3)}+e_{p,k_2}q^{-(1/2-1/p)k_2}M_{2,d}^{(k_1,k_3),(e_1,e_3),(h_1,h_3)}\\ 
&&+e_{p,k_3}q^{-(1/2-1/p)k_3}M_{2, d}^{(k_1,k_2),(e_1,e_2),(h_1,h_2)}+O \left(\delta_{k_1,k_2,k_3} k_1^2q^{-k_1/2}+\delta_{k_1,k_2,2k_3} k_1^2 q^{-3k_1/4}\right) \\ 
&&+ O\left( \delta_{k_1,2k_2,2k_3} k_1^2 q^{-k_1/2}+k_1^2 q^{-k_1/6-k_2-k_3} +q^{-(1/2 - 1/p)(k_1+k_2+k_3)} \right).\\
\end{eqnarray*}
\end{thm}

\begin{proof}
We can use induction in the same way as we used it in the proof of Lemma \ref{lemm3}. The only new term to be considered is given by the case $g_1=g_2=g_3$ and  $p u_1 \mid (e_1h_1k_1+e_2h_2k_2+e_3h_3k_3)$.
This term yields
\[ q^{-(k_1+k_2+k_3)/2}e_{p,e_1h_1k_1+e_2h_2k_2+e_3h_3k_3}\sum_{{{m\mid (k_1,k_2,k_3)}\atop{mp\nmid k_1,k_2,k_3}}\atop{mp \mid (e_1h_1k_1+e_2h_2k_2+e_3h_3k_3)}} \pi(m)m^3.  \\\]
Suppose that $k_1\geq k_2\geq k_3$. If $k_1=k_3$, we have
\[\sum_{{{m\mid k_1}\atop{mp\nmid k_1}}\atop{mp \mid (e_1h_1+e_2h_2+e_3h_3)k_1}}\pi(m)m^3 = O\left(k_1^2q^{k_1}\right).\]
If $k_1=2k_3$, $k_2=k_1$ or $k_2=k_3$, we have
\[\sum_{{{m\mid (k_1,k_2,k_3)}\atop{mp\nmid k_3}}\atop{mp \mid (e_1h_1k_1+e_2h_2k_2+e_3h_3k_3)}} \pi(m)m^3=O\left(k_1^2q^{k_1/2}\right). \]
Finally, for the other cases,
\[\sum_{{{m\mid (k_1,k_2,k_3)}\atop{mp\nmid k_1,k_2,k_3}}\atop{mp \mid (e_1h_1k_1+e_2h_2k_2+e_3h_3k_3)}} \pi(m)m^3=O\left(k_1^2q^{k_1/3}\right). \]

\end{proof}

\begin{thm} Let  $0< h_1, h_2, h_3 \leq (p-1)/2$. For any $K$ with $\max\{1, 1/|\mathcal I|\} <K< d/3$,
 \begin{eqnarray*}
 &&\left< S^\pm(K, f, \psi^{h_1}) S^\pm(K, f, \psi^{h_2})S^\pm(K, f, \psi^{h_3})\right>   
 \\&=& \begin{cases}\frac{3C}{\pi^2}\log(K|\mathcal{I}|)+O\left(1\right) & h_1=h_2= h_3,\\
                                                                                             \frac{C}{\pi^2}\log(K|\mathcal{I}|)+O\left(1\right) & h_{j_1}=h_{j_2}\not = h_{j_3}, \, (j_1,j_2,j_3)\, \text{a permutation of}\, (1,2,3),\\
O(1) & h_i \,\text{distinct}.
                                                                                            \end{cases}
                                                                                            \end{eqnarray*}
where $C$ is the constant defined in Remark $\ref{Chantal's favorite nonconstant}$.
\end{thm}
\begin{cor}
 For any $K$ with $\max\{1, 1/|\mathcal I|\} <K< d/3$,
\[
\langle S^\pm(K,C_f)^3\rangle=\frac{6C(p-1)^2}{\pi^2}\log(K|\mathcal I|)+O(1).
\]
\end{cor}

\section{General Moments}\label{genmom}

Let $n, k_1, \dots, k_n$ be positive integers, let $e_1, \dots, e_n$ take values $\pm 1$ and let $h_1, \dots, h_n$ be integers such that $p\nmid h_i$, $1\leq i\leq n$.
Let ${\bf k} = (k_1, \dots, k_n)$, ${\bf e} = (e_1, \dots, e_n)$ and ${\bf h}=(h_1,\dots, h_n)$. 
Let $\alpha_i \in \F_{q^{k_i}}$, $1 \leq i \leq n$, and let ${\bm \alpha}=(\alpha_1,\dots,\alpha_n)$. We define
$$
m_{n, d}^{{\bf k}, {\bf e}, {\bf h}}({\bm \alpha}) =
\left< \psi(e_1 h_1 \tr_{k_1}f(\alpha_1) + \dots +
e_nh_n \tr_{k_{n}} f(\alpha_n)) \right>
$$
and
$$M_{n, d}^{{\bf k}, {\bf e}, {\bf h}} =
 \sum_{{\alpha_i \in \F_{q^{k_i}}} \atop {i=1, \dots, n}}
q^{-(k_1 + \dots + k_n)/2} m_{n, d}^{{\bf k}, {\bf e}, {\bf h}}({\bm \alpha}).$$

We are computing in this section the general moments
$$\left< S^\pm(K, f, \psi)^n \right>
=  \sum_{k_1, \dots, k_n=1}^K \sum_{e_1, \dots, e_n = \pm 1}
I_K^{\pm}(e_1 k_1) \dots I_K^{\pm}(e_n k_n) M_{n,d}^{{\bf k}, {\bf e}} 
$$
and
\[
\left< S^\pm(K, f, \psi^{h_1})\dots S^\pm(K, f, \psi^{h_n})  \right>=  \sum_{k_1, \dots, k_n=1}^K \sum_{e_j = \pm 1 , \atop {1 \leq j \leq n}}
I_K^{\pm}(e_1 k_1) \dots I_K^{\pm}(e_n k_n) M_{n,d}^{{\bf k}, {\bf e}, {\bf h}}.
\]

\begin{lem} \label{generalcase} Assume $k_1, \dots, k_n > 0$, $k_1 + \dots + k_n < d$.
Let $g_1, \dots, g_s$ of degree $u_1, \dots, u_s$ respectively be all the distinct minimal polynomials over $\F_q$ of
$\alpha_1, \dots, \alpha_n$ (we allow the possibility that some $\alpha_i$'s are conjugate to each other, thus $s\leq n$), and let
$$\epsilon_i = \frac{1}{u_i} \sum_{\alpha_j \in R(g_i)} {k_j} e_j h_j, \;\; 1 \leq i \leq s,$$
where $R(g)$ is the set of roots of $g$.
Then
$$m_{n, d}^{{\bf k}, {\bf e}, {\bf h}}({\bm \alpha}) = \left\{ \begin{array}{ll} 1 & \mbox{if $p \mid \epsilon_i$ for $1 \leq i \leq s$}, \\
0 & \mbox{otherwise}. \end{array} \right.$$
\end{lem}

\begin{proof} As before, we can take the average over the family $\mathcal{F}_d$ of polynomials of degree $d$ without the
condition that $a_{kp} = 0$ for $1 \leq k \leq d/p$.
Renumbering, suppose that $\alpha_i$ has minimal polynomial $g_i$ for $1 \leq i \leq s$.

Since $\sum_{i=1}^s u_i \leq \sum_{i=1}^s k_i < d$, the map
\begin{eqnarray*}
\tau: \mathcal{F}_d \rightarrow \F_q[X]/(g_1 \dots g_s) \simeq \F_{q^{u_1}}\times \dots
\times \F_{q^{u_s}}
\end{eqnarray*}
is exactly $(q-1) q^{d - (u_1+ \dots +u_s)}$-to-one, and as $f$ ranges over $\mathcal{F}_d$,
$\left( f(\alpha_1), \dots , f(\alpha_s) \right)$ takes every value in $\F_{q^{u_1}} \times \dots
\times \F_{q^{u_s}}$ the same number of times. Now, the product 
$\left( \tr_{u_1} f(\alpha_1), \dots, \tr_{u_s}f(\alpha_s) \right)$ also takes every value in $(\F_p)^s$ the same number of times as $f$ ranges over $\mathcal{F}_d$, and the same holds for any linear combination
$$\gamma_1 \tr_{u_1} f(\alpha_1) + \dots + \gamma_s \tr_{u_s}f(\alpha_s),
$$
unless $p$ divides every $\gamma_i$. This shows that each $p$th root of unity occurs
as many times as
$$\psi \left( \gamma_1 \tr_{u_1} f(\alpha_1) + \dots + \gamma_s \tr_{u_s}f(\alpha_s) \right)$$
when $p$ does not divide all the $\gamma_i$.
We now determine the coefficients $\gamma_i$ for
$$m_{n, d}^{{\bf k}, {\bf e}, {\bf h}}({\bm \alpha}) =
\sum_{f \in \mathcal{F}_d} \psi \left( e_1 h_1 \tr_{k_1} f(\alpha_1) + \dots + e_nh_n \tr_{k_n}f(\alpha_n) \right).$$
Recall that $\tr_{k_i}f(\alpha_i)=\frac{k_i}{u_i} \tr_{u_i} f(\alpha_i)$ for $i=1, \dots, s$.
Let
$$\epsilon_i = \frac{1}{u_i} \sum_{\alpha_j \in R(g_i)} e_j h_j k_j, \;\; 1 \leq i \leq s.$$
Then $\gamma_i=\epsilon_i$, i.e.,
$$m_{n, d}^{{\bf k}, {\bf e}, {\bf h}}({\bm \alpha}) =
\sum_{f \in \mathcal{F}_d} \psi \left( \epsilon_1 \tr_{u_1} f(\alpha_1) + \dots + \epsilon_s \tr_{u_s}f(\alpha_s) \right),$$
which implies that $m_{n, d}^{{\bf k}, {\bf e}, {\bf h}}({\bm \alpha})$ takes the value 1 if  $p \mid \epsilon_i$ for $1 \leq i \leq s$, and 0 otherwise.
\end{proof}

Recall that $\pi(m)$ denotes the number of monic irreducible polynomials in $\F_q[X]$.

\begin{lem} \label{generalcase2}
Assume $k_1, \dots, k_n > 0$, $k_1 + \dots + k_n < d$. Then
$M_{n, d}^{{\bf k}, {\bf e}, {\bf h}}$ is bounded  by a sum  of terms made of products of elementary terms of the
type
\[q^{-(j_1+\dots+j_r)/2} \sum_{{{m\mid (j_1,\dots,j_r)}\atop{mp\mid \sum_{i=1}^r e_i h_i j_i}}} \pi(m) m^r \]
where the indices $j_1, \dots, j_r$ of the elementary terms appearing in each product are in bijection with $k_1, \dots, k_n$.

Let $N_{n, d}^{{\bf k}, {\bf e}, {\bf h}}$ be the sum of the terms made exclusively of products of elementary terms
$$q^{-(j_1+j_2)/2} \sum_{{{m\mid (j_1,j_2)}\atop{mp\mid e_1h_1 j_1+ e_2h_2 j_2}}} \pi(m) m^2.
$$
If $n$ is odd, these terms will also be multiplied by an elementary term $$e_{p,j} q^{-j/2} \sum_{{{m\mid j}\atop{mp\mid e  j}}} \pi(m) m
= e_{p,j} \sum_{m \mid \frac{j}{p}} \pi(m)m = e_{p,j} \# \F_{q^{j/p}} = e_{p,j} q^{j/p}.$$
Let $E_{n, d}^{{\bf k}, {\bf e}, {\bf h}}$ be the sum of all the other terms appearing in $M_{n, d}^{{\bf k}, {\bf e}, {\bf h}}$.
Then, $M_{n, d}^{{\bf k}, {\bf e}, {\bf h}} = N_{n, d}^{{\bf k}, {\bf e}, {\bf h}} + O \left( E_{n, d}^{{\bf k}, {\bf e}, {\bf h}} \right).$
\end{lem}

\begin{proof} We first remark that the number of $(\alpha_1, \dots, \alpha_t) \in
\F_{q^{k_1}} \times \dots \times \F_{q^{k_t}}$ which are conjugate over $\F_q$
is
$$\sum_{m \mid (k_1, \dots, k_t)} \pi(m) m^t.$$
Using Lemma \ref{generalcase},
we then have to count the contribution
coming from the ${\bm \alpha} = (\alpha_1, \dots, \alpha_n)$ such that $p \mid \epsilon_i$ for $1 \leq i \leq s$.
Let $\mathcal{P}$ be the set of partitions of  $n$ in $s$ subsets $T_1, \dots, T_s$. Let $k(T_j)$ be the gcd of the $k_i$ such that $i \in T_j$ and let $s(T_j)=\sum_{i\in T_j} e_i h_i k_i$.
Then, for any such partition, the number of ${\bf \alpha} = (\alpha_1, \dots, \alpha_n) \in \F_{q^{k_1}}
\times \dots \times \F_{q^{k_n}}$
such that $\alpha_i$ is a root of $g_j$ when $i \in T_j$ is less than or equal to
$$
\sum_{{{m\mid k(T_1)}\atop{mp \mid s(T_1)}}} \pi(m)m^{|T_1|} \dots
\sum_{{{m\mid k(T_s)}\atop{mp \mid s(T_s)}}} \pi(m)m^{|T_s|} .
$$
This proves the first statement of the lemma. We remark that the above count is an over-count, as it may also count polynomials $g_1, \dots, g_s$ which are not distinct. For example, the number of $(\alpha_1, \alpha_2, \alpha_3, \alpha_4) \in \F_{q^{j_1}} \times \dots \times  \F_{q^{j_4}}$ with minimal polynomials $g_1=g_2, g_3=g_4$ and $g_1 \neq g_3$
 is
\begin{eqnarray*} &&q^{-(j_1+\dots+j_4)/2}\sum_{{{m\mid (j_1,j_2)}\atop{mp\mid e_1 h_1 j_1 + e_2 h_2 j_2}}} \pi(m) m^2 \sum_{{{m\mid (j_3,j_4)}\atop{mp\mid e_3 h_3 j_3 + e_4 h_4 j_4}}} \pi(m) m^2  -  q^{-(j_1+\dots+j_4)/2} \sum_{{{m\mid (j_1,\dots,j_4)}\atop{mp\mid e_1 h_1 j_1 + \dots + e_4 h_4 j_4}}} \pi(m) m^4,
\end{eqnarray*}
which can be written as a term in $N_{n, d}^{{\bf k}, {\bf e}, {\bf h}}$ and a term in $E_{n, d}^{{\bf k}, {\bf e}, {\bf h}}$. The general case
is similar. Suppose that $n=2\ell$ is even. Then, using inclusion-exclusion, the number of $(\alpha_1, \dots, \alpha_n)
\in (\F_{q^{k_1}}, \dots, \F_{q^{k_n}})$
such that $\alpha_i$ and $\alpha_{\ell+i}$ have minimal polynomial $g_i$, and all the $g_i$ are distinct can be written as
\begin{eqnarray*}
&&q^{-(k_1+\dots+k_{2\ell})/2} \left( \sum_{{{m\mid (k_1,k_{\ell+1})\atop{mp\mid e_1 h_1 k_1 + e_{\ell+1} h_{\ell+1}k_{\ell+1}}}}} \pi(m) m^2 \dots
\sum_{{{m\mid (k_{\ell},k_{2 \ell})}\atop{mp\mid e_\ell h_\ell k_\ell + e_{2 \ell} h_{2\ell}k_{e \ell}}}} \pi(m) m^2 \right) + S(k_1, \dots, k_n)\\
\end{eqnarray*}
\noindent where $S(k_1, \dots, k_n)$ is a sum of terms in  $E_{n, d}^{{\bf k}, {\bf e}, {\bf h}}$. 

The case of $n=2\ell+1$ follows similarly, taking into account 
that one has to multiply by the factor $e_{p,k_n} q^{-k_n/2} \sum_{{{m\mid k_n}\atop{mp\mid e  k_n}}} \pi(m) m$.
\end{proof}

We now compute
\[
\left< S^\pm(K, f, \psi^{h_1})\dots S^\pm(K, f, \psi^{h_n})  \right>=\sum_{{k_1, \dots, k_n =1}\atop{e_1, \dots, e_n = \pm 1}}^K 
I_K^{\pm}(e_1 k_1) \dots I_K^{\pm}(e_n k_n) M_{n,d}^{{\bf k}, {\bf e}, {\bf h}}.
\]

We will use $K = cd$ where $0 < c < 1/n$. Then, $k_i \leq K$ implies that $k_1 + \dots + k_n <d$, and we
can apply the lemmas above.

Using Lemma \ref{generalcase2}, we have to compute sums of the
type
\begin{eqnarray} \sum_{k=1}^K \widehat{I}_K^{\pm}(k) q^{-(1/2 - 1/p)k} = C(K)= O(1),
\end{eqnarray}
and for $r \geq 2$
\begin{eqnarray*} \sum_{k_1,\dots,k_r=1}^K\widehat{I}_K^\pm(e_1 k_1) \dots \widehat{I}_K^\pm(e_r k_r)q^{-(k_1+\dots+k_r)/2}\sum_{{m\mid (k_1,\dots,k_r)}\atop{mp \mid  \sum_{i=1}^f e_i h_ik_i}} \pi(m)m^r.
\end{eqnarray*}

If $r=2$, we have when $p \mid e_1 h_1 k_1 + e_2 h_2 k_2$
\begin{eqnarray}  \label{bigones}
&&\sum_{k_1, k_2=1}^K\widehat{I}_K^\pm(e_1 k_1) \widehat{I}_K^\pm(e_2 k_2)
q^{-(k_1+k_2)/2}\sum_{{m\mid (k_1,k_2)}\atop{mp \mid (e_1h_1k_1+e_2h_2k_2)}} \pi(m)m^2 \nonumber \\&&=\left \{\begin{array}{ll} \frac{1}{2 \pi^2} \log{\left(K |\mathcal{I}|\right)} +O(1)&e_1h_1+e_2h_2\equiv 0\, \mathrm{mod}\, p,\\\\O(1) & \text{otherwise}\end{array}\right.
\end{eqnarray}
as we computed in the proof of Theorems \ref{Mcovariance} and \ref{covariance}. (In those theorems we
had the extra condition $mp \nmid k_1,k_2$ in the sum, but those additional terms only add an $O(1)$ to the
final sum, and we can ignore them.)

For the other terms, we have
\begin{lem} \label{rbig} Let $r>2$, then
\[S:=\sum_{k_1,\dots,k_r=1}^K\widehat{I}_K^\pm(k_1) \dots \widehat{I}_K^\pm(k_r)q^{-(k_1+\dots+k_r)/2}\sum_{{m\mid (k_1,\dots,k_r)}\atop{mp\nmid (k_1,\dots,k_r)}} \pi(m)m^r=O(1)\]
\end{lem}
\begin{proof}
Suppose for the moment that $k_1\geq\dots\geq k_r$. If $k_1=k_r$, we have
\[\sum_{{m\mid (k_1,\dots,k_r)}\atop{mp\nmid (k_1,\dots,k_r)}} \pi(m)m^f=O\left(k_1^r q^{k_1} \right).\]
If $k_1=2k_r$, and all the other $k_i$ are equal to $k_1$ or $k_r$, we have
\[\sum_{{m\mid (k_1,\dots,k_r)}\atop{mp\nmid (k_1,\dots,k_r)}} \pi(m)m^r=O\left(k_1^r q^{k_1/2} \right).\]
In all the other cases,
\[\sum_{{m\mid (k_1,\dots,k_r)}\atop{mp\nmid (k_1,\dots,k_r)}} \pi(m)m^r=O\left(k_1^r q^{k_1/3} \right).\]
Putting things together, we get
\begin{eqnarray*}S&\ll &\sum_{k=1}^K\widehat{I}_K^\pm(k)^r k^rq^{-(r-2)k/2}+\sum_{\ell=1}^{r-1}\sum_{k=1}^K\widehat{I}_K^\pm(2k)^\ell  \widehat{I}_K^\pm(k)^{r-\ell}k^rq^{(1-r/2-\ell/2)k}\\
&& +\sum_{k_1,\dots,k_r=1}^K\widehat{I}_K^\pm(k_1) \dots \widehat{I}_K^\pm(k_r)k_1^rq^{-k_1/6-(k_2+\dots+k_r)/2}\\&\ll& 1\end{eqnarray*}
by Proposition \ref{propmanysums}.
\end{proof}

\begin{thm}\label{thm:Smoments} 
For any $K$ with $\max\{1, 1/|\mathcal I|\} <K< d/n$
\[\left< S^\pm(K, f, \psi)^{n} \right>= \left\{\begin{array}{ll}\frac{(2 \ell)!}{\ell! (2\pi^2)^\ell}  \log^\ell(K|\mathcal{I}|) \left(1+O\left(\log^{-1}(K|\mathcal{I}|)\right)\right) & n=2\ell,\\ \\
C\frac{(2\ell+1)!}{\ell! (2\pi^2)^{\ell}}\log^\ell(K|\mathcal{I}|)\left(1+O\left({\log^{-1}\left(K|\mathcal{I}|\right)}\right)\right) & n=2\ell+1, \end{array}\right.\]
where $C$ is defined in Remark $\ref{Chantal's favorite nonconstant}$.
\end{thm}

\begin{proof}
 By Lemmas  \ref{generalcase2} and \ref{rbig}, we observe that the leading term in   $S^{\pm}(K, f, \psi)^n$
will come from the contributions $N_{n, d}^{{\bf k}, {\bf e}}$. By equation (\ref{bigones}), if $n=2\ell$, the leading terms are of the form
$$\left( \frac{1}{2\pi^2} \log{\left( K|\mathcal{I}| \right)} \right)^\ell$$
and if $n=2\ell+1$, the leading terms are of the
form
$$
C \left( \frac{1}{2\pi^2} \log{\left(K |\mathcal{I}|\right)} \right)^\ell.
$$

The final coefficient is obtained by counting the numbers of ways to choose the $\ell$ (or $\ell+1$) coefficients $k_i's$ with positive sign ($e_i=1$) and to pair them with those with negative sign ($e_j=-1$). 

\end{proof}
As $S^\pm(K, f, \psi)=S^\pm(K,f,\bar\psi)$, it is sufficient to study the 
 sum of $S^\pm(K, f, \psi^j)$ for $j$ up to $(p-1)/2$ rather than $p-1$.

We let
\[
\delta_n(C)=\begin{cases} 1 & n=2\ell
\\
C & n=2\ell+1.
\end{cases}
\]
\begin{thm}\label{thm:genmoments}
Let $\ell=\lfloor \frac{n}{2}\rfloor$.  Let $0< h_1,\dots, h_n\leq (p-1)/2$. Then for any $K$ with $\max\{1, 1/|\mathcal I|\} <K< d/n$,
\begin{eqnarray*}&&\left< S^\pm(K, f, \psi^{h_1})\dots S^\pm(K, f, \psi^{h_n})  \right>\\ &=& \delta_n(C)\frac{\Delta(h_1,\dots,h_n)}{(2\pi^2)^\ell}  \log^\ell(K|\mathcal{I}|) \left(1+O\left(\log^{-1}(K|\mathcal{I}|)\right)\right) \end{eqnarray*}
where $C$ is defined in Remark $\ref{Chantal's favorite nonconstant}$ and 
\begin{equation*}\Delta(h_1,\dots,h_{n})=\# \{(e_1,\dots,e_{n})\in \{-1,1\}, \sigma \in \mathbb{S}_{n} | e_1h_{\sigma(1)}+e_2h_{\sigma(2)}\equiv \dots \equiv e_{2\ell-1}h_{\sigma(2\ell-1)}+e_{2\ell}h_{\sigma(2\ell)}\equiv 0 \, \mathrm{mod}\, p  \}\end{equation*} 
where $\mathbb{S}_{n}$ denotes the permutations of the set of $n$ elements.
\end{thm}

\begin{proof}
 By Lemmas  \ref{generalcase2} and \ref{rbig}, we observe that the leading term in  the product $S^\pm(K, f, \psi^{h_1})\dots S^\pm(K, f, \psi^{h_n})$
will come from the contributions $N_{n, d}^{{\bf k}, {\bf e}, {\bf h}}$. By Theorem \ref{covariance}, if $n=2\ell$, the leading terms are of the form
$$\left( \frac{1}{2\pi^2} \log{\left( K|\mathcal{I}| \right)} \right)^\ell$$
and if $n=2\ell+1$, the leading terms are of the
form
$$
C \left( \frac{1}{2\pi^2} \log{\left(K |\mathcal{I}|\right)} \right)^\ell.
$$

The final coefficient is obtained by counting the numbers of ways to choose the $\ell$ (or $\ell+1$) coefficients $k_i$ with positive sign ($e_i=1$) and to pair them with $k_j$ with negative sign ($e_j=-1$) in such a way that
$p$ divides $e_ih_i+e_jh_j$. 

\end{proof}

We note that if $n=2\ell$,
 \begin{equation}\label{combinatorics}
 \sum_{h_1, \dots, h_n=1}^{(p-1)/2} \Delta(h_1,\dots,h_n)=\frac{(p-1)^\ell (2 \ell)!}{2^\ell\ell! }.\end{equation}
There are $\frac{(2\ell)!}{\ell!2^\ell}$ ways of choosing  pairs $\{e_i,e_j\}$ (because the order does not count inside the pair). For each pair either $e_i$ or $e_j$ can be negative and the other one positive so there are a total $2^\ell$ choices for the signs. Finally, for each pair there are $((p-1)/2)$ possible values for  $h_i$ and this determines $h_j$.

\begin{rem} A consequence of Theorem \ref{thm:genmoments} is that the moments are given by sums of products of covariances, exactly in the same way as the moments of a multivariate normal distribution. Moreover, the generating function of the moments converges due to \eqref{combinatorics}. Therefore, our random variables are jointly normal. Since the variables are uncorrelated (cf.~Theorem \ref{covariance}), it follows that our random variables are independent.
\end{rem}

Recall that
\[
S^\pm(K,C_f)=\sum_{j=1}^{p-1}S^{\pm}(K, f, \psi^j).
\]

\begin{thm}\label{thm:sumisgaussian}
Assume that $K=d/\log\log(d|\mathcal I|)$, $d \rightarrow \infty$ and either  $0<|\mathcal{I}|<1$ is fixed or $|\mathcal{I}| \rightarrow 0$
while $d|\mathcal{I}| \rightarrow \infty$.
Then
\[
\frac{S^\pm(K, C_f)}{\sqrt{\frac {2(p-1)}{\pi^2}\log(d|\mathcal{I}|)}}
\]
 has a standard Gaussian limiting distribution when $d \rightarrow \infty$.
\end{thm}
\begin{proof}
First we compute the moments and then we normalize them. 
Let $\ell=\lfloor \frac{n}{2}\rfloor$. We note 
that  with our choice of $K$ we have
\[
\frac{\log (K|\mathcal I|)}{ \log (d|\mathcal I |)}=1- \frac{\log\log\log(d|\mathcal I|)}{\log (d|\mathcal I |)}.\]
Therefore, we can replace $\log (K|\mathcal I|)$ by $\log (d|\mathcal I |)$ in our formulas.
 
Recall that $S^\pm(K, f, \psi^j)=S^\pm(K, f, \psi^{p-j})$, then
\begin{eqnarray*}
 S^\pm(K,C_f)^n=\left(2\sum_{j=1}^{(p-1)/2}S^\pm(K, f, \psi^j) \right)^n=2^n\sum_{j_1, \dots, j_n=1}^{(p-1)/2}S^\pm(K,f,\psi^{j_1})\dots S^\pm(K,f,\psi^{j_n}).
\end{eqnarray*}
Therefore, we can compute the moment
\begin{eqnarray*}
\left\langle S^\pm(K, C_f)^n  \right\rangle 
&=&2^n\sum_{j_1, \dots, j_n=1}^{(p-1)/2}\langle S^\pm(K,f,\psi^{j_1})\dots S^\pm(K,f,\psi^{j_n})\rangle
\end{eqnarray*}
and then by Theorem \ref{thm:genmoments} this is asymptotic to 
\begin{eqnarray*}
&&\frac{2^n\delta_{n}(C)}{(2\pi^2)^\ell}\log^\ell(d|\mathcal{I}|)\sum_{j_1, \dots, j_n=1}^{(p-1)/2} \Delta(j_1,\dots,j_n).
\end{eqnarray*}
Finally we use equation \eqref{combinatorics} to conclude that when $n=2\ell$,
\[
\left\langle S^\pm(K, C_f)^n  \right\rangle=\frac{2^n(p-1)^\ell (2 \ell)!}{2^\ell \ell! (2\pi^2)^\ell}\log^\ell(d|\mathcal{I}|)=\frac{(2\ell)!}{\ell!\pi^{2\ell}} (p-1)^\ell\log^\ell(d|\mathcal{I}|)
\]
and the variance is asymptotic to $\frac{2(p-1)}{\pi^2}\log(d|\mathcal{I}|)$. 

Hence the normalized moment converges to $0$ for $n$ odd and for $n$ even,
\[\lim_{d\rightarrow \infty} \frac{\left\langle S^\pm(K, C_f)^{2\ell} \right\rangle}{\left( \sqrt{\frac{2(p-1)}{\pi^2} \log (d|\mathcal I|)}\right)^{2\ell}}
= \frac{(2\ell)!}{\ell!2^\ell}.\]
\end{proof}

\section{Proof of main theorem}\label{proof}
We prove in this section that 
$$ \frac{N_{\mathcal{I}}(C_f)
- 2g |\mathcal{I}|}{\sqrt{(2(p-1)/\pi^2) \log(d |\mathcal{I}|)}|}$$
converges in mean square to 
$$\frac{S^{\pm}(K, C_f)}{{\sqrt{(2(p-1)/\pi^2) \log(d |\mathcal{I}|)}}}.$$
Then, using Theorem \ref{thm:sumisgaussian}, we get the result of Theorem \ref{mainthm}
since convergence in mean square implies convergence in distribution. 

\begin{lem} Assume that $K =d/\log \log(d|\mathcal{I}|)$, $d\rightarrow  \infty$ and either $0<|\mathcal{I}|<1$ is fixed or
$|\mathcal{I}|\rightarrow 0$ while $d |\mathcal{I}|\rightarrow \infty$. Then
\[\left< \left| \frac{N_\mathcal{I}(C_f) - (d-1)(p-1) |\mathcal{I}| +S^{\pm}(K,C_f)}{\sqrt{(2(p-1)/\pi^2) \log (d |\mathcal{I}|)}}\right|^2 \right>\rightarrow 0\]
\end{lem}
\begin{proof}
From  equation (\ref{T-estimate}) from Section \ref{1mom}, using the Beurling-Selberg polynomials and the explicit formula (Lemma \ref{Explicit-Formula}),
we deduce that
\begin{eqnarray*} 
 \frac{-(p-1)(d-1)}{K+1} &\leq& N_\mathcal{I}(C_f) - (p-1)(d-1) |\mathcal{I}|  +S^-(K, C_f)  \\& \leq &S^-(K, C_f)
- S^+(K,C_f)   +\frac{(p-1)(d-1)}{K+1} \end{eqnarray*}
and
\begin{eqnarray*} 
 \frac{-(p-1)(d-1)}{K+1} &\leq& -N_\mathcal{I}(C_f) + (p-1)(d-1) |\mathcal{I}|  -S^+(K, C_f)   \\ &\leq& S^-(K, C_f)
- S^+(K,C_f)   + \frac{(p-1)(d-1)}{K+1}.\ \end{eqnarray*}

Using these two inequalities to bound the absolute value of the central term, we obtain
\begin{eqnarray*}
&&\left< \left ( N_\mathcal{I}(C_f) - (p-1)(d-1) |\mathcal{I}| + S^{\pm}(K, C_f) \right)^2 \right>\\
&\leq& \max \left \{ \left(\frac{(p-1)(d-1)}{K+1} \right)^2, \left< \left(S^{-}(K, C_f) - S^{+}(K,C_f) +\frac{(p-1)(d-1)}{K+1}\right)^2 \right> \right\} \\
&&\leq  \left(\frac{(p-1)(d-1)}{K+1} \right)^2 \\&+& \max \left \{ 0,  \left< \left( S^{-}(K, C_f) - S^{+}(K, C_f) \right)^2 \right> 
+ 2 \frac{(p-1)(d-1)}{K+1} \left<  S^{-}(K, C_f) - S^{+}(K, C_f) \right> \right\}.
\end{eqnarray*}

Now using the estimate in the proof of Theorem \ref{averageET}, we have that
\begin{eqnarray*}
 \left\langle S^{-}(K, C_f) - S^{+}(K, C_f) \right\rangle &=& \left< S^{-}(K, C_f)  \right> - 
\left<   S^{+}(K, C_f) \right> =O(1).
\end{eqnarray*}
For the remaining term we note that 
\begin{eqnarray*}
&&\left\langle \left(S^{-}(K, C_f) - S^{+}(K, C_f)\right)^2\right \rangle
\\
&=&\left\langle \left(S^{-}(K, C_f)\right)^2\right\rangle + \left\langle \left(S^{+}(K, C_f)\right)^2\right\rangle-2\left\langle  \sum_{j_1, j_2=1}^{p-1} S^{-}(K, f, \psi^{j_1}) S^{+}(K, f, \psi^{j_2})\right\rangle.
\end{eqnarray*}
By Corollary \ref{cor:2ndmoment}, this equals
\[\frac{4(p-1)}{\pi^2}\log(d|\mathcal I|)+O(1)-\frac{4(p-1)}{\pi^2}\log(d|\mathcal I|)+O(1)=O(1).\]
Therefore, 
\[\left< \left ( N_\mathcal{I}(C_f )- (p-1)(d-1) |\mathcal{I}| +S^{\pm}(K, C_f) \right)^2 \right>=O\left(\left(\frac{(p-1)(d-1)}{K+1}\right)^2\right)\]
and
\[\left \langle \left( \frac{N_\mathcal{I}(C_f) - (p-1)(d-1) |\mathcal{I}|+S^\pm(K, C_f)}{\sqrt{(2(p-1)/\pi^2) \log(d|\mathcal{I}|)} }\right)^2 \right \rangle \rightarrow 0\]
when $d$ tends to infinity and $K =d/\log \log(d|\mathcal{I}|)$.
\end{proof}

\begin{rem}
 Proposition \ref{prop} is proved in a similar way. For this, one uses Theorem \ref{thm:Smoments} to examine the moments of 
\[
\frac{S^\pm(K, f, \psi)+S^\pm(K, f, \bar{\psi})}{\sqrt{\frac {4}{\pi^2}\log(d|\mathcal{I}|)}}=\frac{2S^\pm(K, f, \psi)}{\sqrt{\frac {4}{\pi^2}\log(d|\mathcal{I}|)}}.
\]
\end{rem}

\section*{Acknowledgments} 
The authors would like to thank Ze\'ev Rudnick for many useful discussions
while preparing this paper. The authors are also grateful to Louis-Pierre Arguin, Andrew Granville and Rachel Pries for helpful conversations related to this work. The first, third and fifth named authors thank the Centre de Recherche Math\'ematique (CRM)  and the Mathematical Sciences Research Institute (MSRI) for their hospitality. 

This work was supported by the National Science Foundation of U.S. [DMS-1201446 to B. F.],
  the Simons Foundation [\#244988 to A. B.] the UCSD Hellman Fellows Program 
[2012-2013 Hellman Fellowship to A. B.], the Natural Sciences and Engineering Research Council
 of Canada [Discovery Grant 155635-2008 to C. D., 355412-2008 to M. L.] and the Fonds de recherche du Qu\'ebec - Nature et technologies [144987 to M. L., 166534 to C. D. and M. L.]

\end{document}